\documentclass{amsart} 
\usepackage{amsmath, amssymb, verbatim, amscd,amsthm,mathrsfs,mathtools}
\usepackage[normalem]{ulem}
\usepackage[driverfallback=dvipdfm]{hyperref}
\usepackage[dvipdfm]{graphicx}
\usepackage{color,graphicx,shortvrb}
\usepackage{enumerate}
\usepackage[all]{xy}

\newtheorem{thm}{Theorem}[section]

\newtheorem{lem}[thm]{Lemma}
\newtheorem{prop}[thm]{Proposition} 
\newtheorem{cor}[thm]{Corollary}

\theoremstyle{definition}

\newtheorem*{rem}{Remark}
\newtheorem*{nota}{Notation} 
 
\numberwithin{equation}{section}

\newcommand{\qq}{\qquad}

\newcommand{\cxr}{C_0^{+}(X,\mathbb{R})}
\newcommand{\cyr}{C_0^{+}(Y,\mathbb{R})}

\newcommand{\N}{\mathbb N} 
\newcommand{\R}{\mathbb R} 
\newcommand{\fr}{\frac} 
\newcommand{\V}[1]{\left\Vert #1 \right\Vert}

\def\Gam{\Gamma}

\newcommand{\set}[1]{\left\{ #1 \right\}}

\newcommand{\PA}[1]{P_{#1}}
\newcommand{\PB}[1]{Q_{#1}}
\newcommand{\MX}[1]{M_{#1}}
\newcommand{\MY}[1]{N_{#1}}

\newcommand{\pcx}{C_0^+(X,\R)}
\newcommand{\pcy}{C_0^+(Y,\R)}
\newcommand{\pcl}{C_0^+(L,\R)}
\newcommand{\cx}{C_0(X,\R)}
\newcommand{\cy}{C_0(Y,\R)}
\newcommand{\cl}{C_0(L,\R)}
\newcommand{\Vinf}[1]{\V{#1}_\infty}
\newcommand{\zx}{\mathbf{0}_X}
\newcommand{\zy}{\mathbf{0}_Y}
\newcommand{\zl}{\mathbf{0}_L}
\newcommand{\TI}{S}
\newcommand{\TT}{\widetilde{T}} 
\newcommand{\al}{\alpha}

\begin{document}


\title[Phase-isometries between positive cones]
{Phase-isometries between the positive cones of the Banach space
of continuous real-valued functions}


\author[D.~Hirota]{Daisuke Hirota}
\address[Daisuke Hirota]{Graduate School of Science and Technology,
Niigata University, Niigata 950-2181, Japan}
\email{hirota@m.sc.niigata-u.ac.jp}

\author[I.~Matsuzaki]{Izuho Matsuzaki}
\address[Izuho Matsuzaki]{Graduate School of Science and Technology,
Niigata University, Niigata 950-2181, Japan}
\email{matsuzaki@m.sc.niigata-u.ac.jp}

\author[T.~Miura]{Takeshi Miura}
\address[Takeshi Miura]{Department of Mathematics, Faculty of Science, Niigata University, 
Niigata 950-2181, Japan}
\email{miura@math.sc.niigata-u.ac.jp}



\begin{abstract} 
For a locally compact Hausdorff space $L$,
we denote by $\cl$ the Banach space of all continuous real-valued 
functions on $L$ vanishing at infinity equipped with the supremum norm. 
We prove that 
every surjective phase-isometry $T\colon\pcx\to \pcy$ between the positive cones of $\cx$ and $\cy$
is a composition operator induced by a homeomorphism
between $X$ and $Y$.
Furthermore, we show that any surjective phase-isometry
$T\colon\pcx\to\pcy$ extends to a surjective linear isometry from $\cx$ onto $\cy$.
\end{abstract}

\maketitle


\section{Introduction and main theorem}

The celebrated Wigner's unitary-antiunitary theorem
\cite{wig} characterized surjective, not necessarily
linear, maps $T\colon H_1\to H_2$ between two complex
inner product spaces $H_1$ and $H_2$ with
the property that 
\begin{equation}\label{1-1}
|\langle T(x),T(y)\rangle|
=|\langle x,y\rangle|\qquad(x,y\in H_1).
\end{equation}
There are several proofs of the theorem;
see, for example, \cite{Bargmann, geh, gyo, lomont}.
R\"{a}tz \cite{raz} considered a surjective map
$T\colon H_1\to H_2$ between two
real inner product spaces $H_1$ and $H_2$
that satisfies \eqref{1-1}.
Then it was proved in \cite[Corollary~8]{raz} that
there exist a function
$\varphi\colon H_1\to\set{-1,1}$ and a surjective
linear isometry $U\colon H_1\to H_2$
such that $T(x)=\varphi(x)U(x)$ for all $x\in H_1$.

Maksa and P\'ales \cite{mak}
gave some equivalent conditions with
\eqref{1-1}
and proved that the condition \eqref{1-1}
is equivalent to the following:
\begin{equation}\label{phase}
\set{\|T(x)+T(y)\|, \|T(x)-T(y)\|}
=\set{\|x+y\|,\|x-y\|}\qquad(x,y\in H_1).
\end{equation}
One can define maps $T\colon H_1\to H_2$
that satisfy \eqref{phase},
provided that $H_1$ and $H_2$ are both real normed spaces:
Such maps are said to be \textit{phase-isometries}.
Surjective phase-isometries are characterized
for some concrete normed spaces
\cite{XuTa, jia, LiTa,TaXu, ZeHu}.
Ili\v{s}evi\'{c}, Omladi\v{c} and Turn\v{s}ek 
\cite[Theorem~4.2]{ili} gave the characterization
of surjective phase-isometries between
real normed spaces.
Tan and Gao \cite{tan3} and Tan, Zhang and Huang \cite{TZH}
considered surjective phase-isometries between the unit spheres
of certain normed spaces.

Sun, Sun and Dai \cite{sun}
considered phase-isometries on the positive cone,
$C^+(L,\R)$, of all non-negative functions
of the Banach space $C(L,\R)$ of all continuous 
real-valued functions on a compact Hausdorff space $L$.
In \cite[Corollary~2.5]{sun},
they showed that every surjective phase-isometry
between two positive cones 
$C^{+}(K,\mathbb{R})$ and $C^{+}(L,\R)$ 
is the restriction of a surjective linear isometry
between $C(K,\R)$ and $C(L,\R)$
under an additional assumption on $K$.
Similar result was obtained in \cite{sun2} for the positive cone
of the sequence space $c_0$.

Motivated by the above results, it is natural to discuss
phase-isometries between the positive cones, $\pcl$,
of the Banach space $\cl$ of all continuous real-valued
functions on a locally compact Hausdorff space $L$
vanishing at infinity equipped with the supremum norm
$\Vinf{\cdot}$.
The purpose of this paper is to characterize surjective
phase-isometries between two positive cones $\pcx$ and $\pcy$,
which generalizes \cite[Theorem~2.7]{sun}.
The main result of this paper is as follows.

\begin{thm}\label{thm1}
Let $T\colon\pcx\to\pcy$ be a surjective phase-isometry.
Then there exists a homeomorphism $\phi\colon Y\to X$
such that
\[
T(f)(y)=f(\phi(y))
\]
holds for all $f\in\pcx$ and $y\in Y$.
\end{thm}

\begin{cor}\label{cor1}
If $T\colon\pcx\to\pcy$ is a surjective phase-isometry,
then there exists a unique surjective linear isometry
$\TT\colon\cx\to\cy$ such that $\TT=T$ on $\pcx$.
\end{cor}

\section{Preliminaries and auxiliary lemmata}

We denote by $\cl$ the real Banach space of all continuous real-valued functions on a locally compact Hausdorff space $L$
vanishing at infinity with respect to the supremum norm
$\Vinf{h}=\sup_{x\in L}|h(x)|$ for $h\in\cl$.
The positive cone, $\pcl$, is the set of all non-negative functions
of $\cl$, that is,
\[
\pcl=\set{h\in\cl:h\geq0}.
\]
Denote by $\zl$ the zero function in $\pcl$.

First, we shall show that every phase-isometry
between $\pcl$ is an isometry, which was proved
in \cite[Theorem~2.3]{sun} for a compact Hausdorff space
$L$ with an additional assumption on $L$.

\begin{prop}\label{prop1.1}
Let $X$ and $Y$ be locally compact Hausdorff spaces.
If $T\colon\pcx\to\pcy$ is a phase-isometry,
then  it is an isometry.
\end{prop}

\begin{proof}
Fix $f,g\in\pcx$ and $x\in X$ arbitrarily.
We have
\[
|f(x)-g(x)|
\leq|f(x)+g(x)|
\leq\Vinf{f+g},
\]
since $f(x),g(x)\geq0$.
This implies that $\Vinf{f-g}\leq\Vinf{f+g}$. By the same reasoning, we have $\Vinf{T(f)-T(g)}\leq\Vinf{T(f)+T(g)}$. 
Thus, we obtain

\begin{align*}
\Vinf{f-g}
&=\min\set{\Vinf{f-g}, \Vinf{f+g}}\\
&=\min\set{\Vinf{T(f)-T(g)}, 
\Vinf{T(f)+T(g)}}=\Vinf{T(f)-T(g)},
\end{align*}
since $T$ is a phase-isometry. This shows that $T$ is an isometry.
\end{proof}

Once we characterize isometries between the positive cones
$\pcl$,
we will get the structure of phase-isometries between them
by Proposition~\ref{prop1.1}.
In this section, we will analyze surjective isometries
$T\colon\pcx\to\pcy$.
We will prove the following result.

\begin{thm}\label{thm2}
Let $T\colon\pcx\to\pcy$ be a surjective isometry.
Then there exists a homeomorphism $\phi\colon Y\to X$
such that
\[
T(f)(y)=f(\phi(y))
\]
holds for all $f\in\pcx$ and $y\in Y$.
\end{thm}

In the rest of this section, we assume that $T\colon\pcx\to\pcy$
is a surjective isometry.
Then we see that $T^{-1}\colon\pcy\to\pcx$ is a well-defined
surjective isometry:
For simplicity of notation, we shall write $T^{-1}=S$,
and then $\TI\colon\pcy\to\pcx$ is a surjective isometry.
We set
\[
t_0=T(\zx)
\qq\mbox{and}\qq
s_0=\TI(\zy),
\]
and then $t_0\in\pcy$ and $s_0\in\pcx$ by definition. Immediately, we have
\begin{equation}\label{norm1}
\Vinf{T(f)-t_0}=\Vinf{f},
\qq
\Vinf{\TI(u)-s_0}=\Vinf{u}
\end{equation}
for all $f\in\pcx$ and $u\in\pcy$.
We need some lemmata to prove that $t_0=\zy$.

\begin{lem}\label{lem2.1}
If $u\in\pcy$ satisfies $\Vinf{s_0}<\Vinf{u}$,
then $\Vinf{u}\leq\Vinf{\TI(u)}$.
\end{lem}

\begin{proof}
Fix an arbitrary $u\in\pcy$ that satisfies
$\Vinf{s_0}<\Vinf{u}$.
For each $x\in X$, we have
\[
|\TI(u)(x)-s_0(x)|
\leq\max\set{\TI(u)(x),s_0(x)}
\leq\max\set{\Vinf{\TI(u)},\Vinf{s_0}},
\]
since $\TI(u)(x),s_0(x)\geq0$.
We derive from the above inequalities with \eqref{norm1} that
\begin{equation}\label{lem2.1.1}
\Vinf{u}=\Vinf{\TI(u)-s_0}
\leq
\max\set{\Vinf{\TI(u)},\Vinf{s_0}}.
\end{equation}
If we suppose that
$\max\set{\Vinf{\TI(u)},\Vinf{s_0}}=\Vinf{s_0}$,
then we would deduce from \eqref{lem2.1.1} that
$\Vinf{u}\leq\Vinf{s_0}$,
which contradicts the assumption that
$\Vinf{s_0}<\Vinf{u}$.
Therefore,
$\max\set{\Vinf{\TI(u)},\Vinf{s_0}}=\Vinf{\TI(u)}$.
We conclude that $\Vinf{u}\leq\Vinf{\TI(u)}$
by \eqref{lem2.1.1}.
\end{proof}

\begin{lem}\label{lem2.2}
The identity $\Vinf{\TI(u)}=\Vinf{u}$ holds
for all $u\in\pcy$ with $\max\set{\Vinf{t_0},\Vinf{s_0}}<\Vinf{u}$.
\end{lem}

\begin{proof}
Fix an arbitrary $u\in\pcy$
with $\max\set{\Vinf{t_0},\Vinf{s_0}}<\Vinf{u}$.
Lemma~\ref{lem2.1} shows that
$\Vinf{u}\leq\Vinf{\TI(u)}$,
since $\Vinf{s_0}<\Vinf{u}$ by the choice of $u$.

We may apply Lemma~\ref{lem2.1} to $f\in\pcx$
with $\Vinf{t_0}<\Vinf{f}$, instead of $u\in\pcy$ with
$\Vinf{s_0}<\Vinf{u}$.
Then we have $\Vinf{f}\leq\Vinf{T(f)}$.
Since $\Vinf{t_0}<\Vinf{u}\leq\Vinf{S(u)}$
by the first part of this proof,
we may apply the inequality
$\Vinf{f}\leq\Vinf{T(f)}$ to $f=\TI(u)$. It follows that
$\Vinf{\TI(u)}\leq\Vinf{T(\TI(u))}=\Vinf{u}$,
where we have used that $\TI=T^{-1}$.
We conclude that $\Vinf{S(u)}=\Vinf{u}$.
\end{proof}

Now we are in a position to prove that $t_0=\zy$,
which assures us
that the isometry $T$ is a norm preserving map.

\begin{lem}\label{lem2.3}
The identities $t_0=\zy$  and $s_0=\zx$ hold.
Furthermore, $\Vinf{T(f)}=\Vinf{f}$
for all $f\in\pcx$ and $\Vinf{S(u)}=\Vinf{u}$
for all $u\in\pcy$.
\end{lem}

\begin{proof}
Suppose, on the contrary, that $t_0(y_0)\neq0$
for some $y_0\in Y$.
Setting $3p=\Vinf{t_0}$, we see that $p>0$.
Define two subsets, $Y_0$ and $Y_1$, of $Y$ as
\[
Y_0=\set{y\in Y:t_0(y)\leq p}
\qq\mbox{and}\qq
Y_1=\set{y\in Y:2p\leq t_0(y)}.
\]
Then
$Y_0$ is a closed subset of $Y$
and $Y_1$ is a compact subset of $Y$,
since $t_0\in\pcy$.
Setting $q=4p+\Vinf{s_0}$,
there exists $u\in\pcy$ such that
$u=0$ on $Y_0$, 
$u=q$ on $Y_1$ and $u(Y)\subset[0,q]$
by Urysohn's lemma.
Fixing an arbitrary $y\in Y$,
we shall prove that $|u(y)-t_0(y)|\leq q-p$
by considering three specific cases.
\begin{itemize}
\item
If $y\in Y_0$, then $u(y)=0$ and $t_0(y)\leq p<q-p$
by the choice of $q$, and thus $|u(y)-t_0(y)|<q-p$.

\item
If $y\in Y_1$, then $u(y)=q$ and $2p\leq t_0(y)$,
which shows that
$|u(y)-t_0(y)|\leq q-2p$.

\item
If $y\in Y\setminus(Y_0\cup Y_1)$,
then $0\leq u(y)\leq q$ and $p<t_0(y)<2p$.
We have
\[
|u(y)-t_0(y)|\leq\max\set{2p,q-p}=q-p,
\]
since $q-p= 3p+\Vinf{s_0}>2p$.
\end{itemize}
The above arguments show that $\Vinf{u-t_0}\leq q-p$.

Notice that
$\max\set{\Vinf{t_0},\Vinf{s_0}}<q=\Vinf{u}$
by the choice of $p$ and $q$.
Applying Lemma~\ref{lem2.2} to $u$, we see that
$\Vinf{\TI(u)}=\Vinf{u}$.
Since $\TI=T^{-1}$ is an isometry,
we obtain $\Vinf{u-t_0}=\Vinf{\TI(u)}=\Vinf{u}=q$
by the last equality,
where we have used that $S(t_0)=S(T(\zx))=\zx$.
Therefore, we obtain
$q=\Vinf{u-t_0}\leq q-p$,
which is impossible because $p>0$.
We thus conclude that $t_0(y)=0$ for all $y\in Y$,
which proves that $T(\zx)=t_0=\zy$. Combining \eqref{norm1} with the last equality, it follows that $\Vinf{T(f)}=\Vinf{T(f)-t_0}=\Vinf{f}$
for all $f\in\pcx$.

Applying the same argument to $S$, we have also that $s_0=\zx$ and $\Vinf{S(u)}=\Vinf{u}$
for all $u\in\pcy$.
\end{proof}

\begin{rem}
We should note that Theorem~\ref{thm2} is a special case of
\cite[Theorem~3.6]{Pe}:
On one hand, Peralta \cite{Pe} showed that every surjective isometry
between two unit spheres of all positive elements of $C^*$-algebras
can be extended to a surjective linear isometry between whole spaces.
On the other hand, we see that
every surjective phase-isometry $T$
between two  positive cones $\pcx$ and $\pcy$
is a norm-preserving surjective isometry 
by Proposition~\ref{prop1.1} and Lemma \ref{lem2.3}.
This implies that the restriction of $T$ to the unit sphere of $\pcx$
is a surjective isometry between
unit spheres of $\pcx$ and $\pcy$.
Then $T$ can be extended to a surjective linear isometry
between whole spaces by \cite[Theorem~3.6]{Pe}.
In this paper, we will give a self-contained proof of
Theorem~\ref{thm2}.
\end{rem}

We employ the peaking function argument
to avoid the G\^ateaux differential,
which was essential in \cite{sun}.
We provide some notations as below.

\begin{nota}
For each $x\in X$ and $y\in Y$,
we define two subsets $\PA{x}$ of $\pcx$
and $\PB{y}$ of $\pcy$ as
\begin{align*}
\PA{x}
&=
\set{f\in\pcx\setminus\set{0}:f(x)=\Vinf{f}},\\
\PB{y}
&=
\set{u\in\pcy\setminus\set{0}:u(y)=\Vinf{u}}.
\end{align*}
For each $f\in \cxr$ and $u\in\cyr$,
we define $\MX{f}$ and $\MY{u}$ as
\[
\MX{f}=\set{x\in X:f(x)=\Vinf{f}}
\qq\mbox{and}\qq
\MY{u}=\set{y\in Y:u(y)=\Vinf{u}}.
\]
\end{nota}

The following simple lemma plays an important role
when we prove uniqueness of peak points.

\begin{lem}\label{lem2.4}
Let $x_0,x_1\in X$.
If $\PA{x_0}\subset\PA{x_1}$,
then $x_0=x_1$.
\end{lem}

\begin{proof}
Suppose, on the contrary, that $x_0\neq x_1$.
Then there exists $f\in\PA{x_0}$ such that
$f(x_0)=1$, $f(x_1)=0$ and $f(X)\subset[0,1]$.
By the choice of $f$, we obtain
$f\in\PA{x_0}\setminus\PA{x_1}$,
which contradicts $\PA{x_0}\subset\PA{x_1}$.
We thus conclude that $x_0=x_1$.
\end{proof}

We will construct a mapping from $Y$ to $X$.
To this end, we first prove that every function
in $\TI(\PB{y})$ attains its maximum at the same point
of $X$ after proving some lemmata.

\begin{lem}\label{lem2.5}
For each $y_0\in Y$, $n\in\N$ and
$u_k\in\PB{y_0}$ with $0\leq k\leq n$,
the sum $v=\sum_{k=0}^nu_k$ satisfies
$v\in\PB{y_0}$ and
\begin{equation}\label{lem2.5.1}
\TI(v)(x_1)=v(y_0)
\end{equation}   
for some $x_1\in X$.
\end{lem}

\begin{proof}
Fix $y_0\in Y$, $n\in\N$ and $u_k\in\PB{y_0}$
with $0\leq k\leq n$ arbitrarily.
Setting $v=\sum_{k=0}^nu_k$, we have
\[
0\leq\sum_{k=0}^nu_k(y_0)
=v(y_0)
\leq\Vinf{v}
\leq\sum_{k=0}^n\Vinf{u_k}
=\sum_{k=0}^nu_k(y_0)
=v(y_0),
\]
since $u_k\geq0$ and $u_k\in\PB{y_0}$.
This shows that $v(y_0)=\Vinf{v}$,
and thus, $v\in\PB{y_0}$.
We note that $\MX{\TI(v)}\neq\emptyset$,
since $\TI(v)\in\pcx$.
Then there exists $x_1\in\MX{\TI(v)}$,
that is, $\TI(v)(x_1)=\Vinf{\TI(v)}$.
It follows from Lemma~\ref{lem2.3} that
$\TI(v)(x_1)=\Vinf{\TI(v)}=\Vinf{v}=v(y_0)$.
\end{proof}

We will further deepen the above argument.

\begin{lem}\label{lem2.6}
For each $y_0\in Y$, $n\in\N$ and
$u_k\in\PB{y_0}$ with $0\leq k\leq n$,
the intersection $\cap_{k=0}^n\MX{\TI(u_k)}$
is a non-empty set.
\end{lem}

\begin{proof}
Let $u_k\in\PB{y_0}$ for $0\leq k\leq n$
and $v=\sum_{k=0}^nu_k$. By Lemma~\ref{lem2.5}, 
we see that $v\in\PB{y_0}$ and 
there exists $x_1\in X$ such that \eqref{lem2.5.1} holds. Fixing an arbitrary $m$ with $0\leq m\leq n$,
we shall prove that $x_1\in\MX{\TI(u_m)}$.
We set $v_m=v-u_m$.
We may apply Lemma~\ref{lem2.5} to $v_m$
to obtain $v_m\in\PB{y_0}$,
since $v_m$ is the finite sum of functions
$u_k$ with $k\in\set{0,\dots,n}\setminus\set{m}$.
Hence, $v_m(y_0)=\Vinf{v_m}$.
We obtain
\[
\TI(v)(x_1)-\TI(u_m)(x_1)
\leq\Vinf{\TI(v)-\TI(u_m)}
=\Vinf{v-u_m}
=\Vinf{v_m}
=v_m(y_0),
\]
because $\TI$ is an isometry.
Thus, $\TI(v)(x_1)-v_m(y_0)\leq\TI(u_m)(x_1)$,
which shows that
\[
\Vinf{u_m}
=u_m(y_0)
=v(y_0)-v_m(y_0)
=\TI(v)(x_1)-v_m(y_0)
\leq\TI(u_m)(x_1)
\leq\Vinf{\TI(u_m)},
\]
where we have used that \eqref{lem2.5.1}.
It follows from Lemma~\ref{lem2.3} that
$\Vinf{\TI(u_m)}=\Vinf{u_m}$,
and therefore, $\TI(u_m)(x_1)=\Vinf{\TI(u_m)}$
by the above inequalities.
This proves that $x_1\in\MX{\TI(u_m)}$ as desired.
\end{proof}

Now we are in a position to prove that each function
in $\TI(\PB{y})$ attains its maximum at the same point.

\begin{lem}\label{lem2.7}
For each $y_0\in Y$,
the intersection $\cap_{f\in\TI(\PB{y_0})}\MX{f}$
is a non-empty set.
\end{lem}

\begin{proof}
Fix $y_0\in Y$ and $f_0\in\TI(\PB{y_0})$ arbitrarily,
and then
$\cap_{f\in\TI(\PB{y_0})}\MX{f}
=\cap_{f\in\TI(\PB{y_0})}(\MX{f}\cap\MX{f_0})$.
We see that $\MX{f_0}$ is a compact set,
since $f_0\in\pcx$,
and thus $\MX{f}\cap\MX{f_0}$ is a closed subset
of the compact set $\MX{f_0}$ for each $f\in\TI(\PB{y_0})$.
Hence, we may apply the finite intersection property
to show that $\cap_{f\in\TI(\PB{y_0})}(\MX{f}\cap\MX{f_0})
\neq\emptyset$.
Choose arbitrary $n\in\N$ and $f_k\in\TI(\PB{y_0})$
with $1\leq k\leq n$.
We need to prove that
$\cap_{k=1}^n(\MX{f_k}\cap\MX{f_0})
=\cap_{k=0}^n\MX{f_k}\neq\emptyset$.
There exists $u_k\in\PB{y_0}$ such that
$f_k=\TI(u_k)$ for each $0\leq k\leq n$,
since $f_k\in\TI(\PB{y_0})$.
We derive from Lemma~\ref{lem2.6} that
$\cap_{k=0}^n\MX{f_k}
=\cap_{k=0}^n\MX{\TI(u_k)}\neq\emptyset$,
as desired.
\end{proof}

Next, we shall prove that each $\PB{y}$ corresponds
to one and only one $\PA{x}$.

\begin{lem}\label{lem2.8}
For each $y_0\in Y$, there exists a unique $x_0\in X$
such that $S(\PB{y_0})=\PA{x_0}$.
\end{lem}

\begin{proof}
Take an arbitrary $y_0\in Y$.
There exists $x_0\in\cap_{f\in\TI(\PB{y_0})}\MX{f}$
by Lemma~\ref{lem2.7}.
This implies that $f\in\PA{x_0}$ 
for all $f\in \TI(\PB{y_0})$,
and consequently, $\TI(\PB{y_0})\subset\PA{x_0}$.
We may apply the last argument and Lemma~\ref{lem2.7}
to the pair of $(T,x_0)$, instead of $(S,y_0)$.
Then there exists $y_1\in Y$ such that
$T(\PA{x_0})\subset\PB{y_1}$,
and thus $\PA{x_0}\subset T^{-1}(\PB{y_1})=\TI(\PB{y_1})$.
By combining the first part of this proof,
we obtain $\TI(\PB{y_0})\subset\PA{x_0}\subset\TI(\PB{y_1})$,
and hence $\PB{y_0}\subset\PB{y_1}$.
We may apply Lemma~\ref{lem2.4} to the pair of
$(\PB{y_0},\PB{y_1})$, instead of $(\PA{x_0},\PA{x_1})$,
and then we obtain $y_0=y_1$.
Therefore, we conclude that
$\TI(\PB{y_0})=\PA{x_0}$.

If $x_1\in X$ satisfies $S(\PB{y_0})=\PA{x_1}$,
then we get $\PA{x_0}=\PA{x_1}$.
Lemma~\ref{lem2.4} shows that $x_0=x_1$,
and thus such $x_0$ is uniquely determined.
\end{proof}

The following lemma plays a key role in generalizing
\cite[Theorem~2.7]{sun}.
The idea of our proof of it is based on
\cite[Lemma~2.17]{cue}.

\begin{lem}\label{lem2.9}
For each $u\in\pcy$ and $y_0\in Y$,
there exists $v\in\PB{y_0}$ such that
$u+v\in\PB{y_0}$.
\end{lem}

\begin{proof}
Taking arbitrary $u\in\pcy$ and $y_0\in Y$,
we define subsets $Y_0$ and $Y_n$ of $Y$ as
\begin{align}\label{lem2.9.1}
\begin{split}
Y_0
&=
\set{y\in Y:\fr{3\Vinf{u}}{4}\leq|u(y)-u(y_0)|},\\
Y_n
&=
\set{y\in Y:\fr{3\Vinf{u}}{2^{n+2}}
\leq|u(y)-u(y_0)|
\leq\fr{3\Vinf{u}}{2^{n+1}}}
\end{split}
\end{align}
for each $n\in\N$.
We see that $Y_k$ is a closed subset of $Y$
with $y_0\not\in Y_k$ for all $k\in\N\cup\set{0}$.
Choose $u_n\in\PB{y_0}$ satisfying 
\begin{equation}\label{lem2.9.2}
u_n(y_0)=3\Vinf{u}
\qq\mbox{and}\qq
u_n(Y_0\cup Y_n)=\set{0}
\end{equation}
for each $n\in\N$. 
We set $v=\sum_{n=1}^\infty u_n/2^n$.
Since $u_n\in \PB{y_0}$ with
$\Vinf{u_n}=3\Vinf{u}$ for all $n\in\N$, the series $\sum_{n=1}^\infty u_n/2^n$
converges in $\pcy$ and 
$\Vinf{v}=3\Vinf{u}$. Fixing an arbitrary $y\in Y$, we shall prove that
$u(y)+v(y)\leq u(y_0)+3\Vinf{u}$
by considering three specific cases.
\begin{itemize}
\item
If $y\in Y_0$, then $v(y)=0$ by \eqref{lem2.9.2}. This implies  that
$u(y)+v(y)\leq \Vinf{u}\leq u(y_0)+3\Vinf{u}$.

\item
If $y\in Y_m$ for some $m\in\N$, then 
we see that
$u(y)\leq u(y_0)+3\Vinf{u}/2^{m+1}$
and $u_m(y)=0$
 by \eqref{lem2.9.1} and \eqref{lem2.9.2}.
Applying $u_m(y)=0$, we get
\[
v(y)
=\sum_{n=1}^\infty\fr{u_n(y)}{2^n}
=\sum_{n\neq m}\fr{u_n(y)}{2^n}
\leq3\Vinf{u}\left(1-\fr{1}{2^m}\right),
\]
since $\Vinf{u_n}=3\Vinf{u}$ for all $n\in\N$.
By combining the last two inequalities, we obtain
\[
u(y)+v(y)
\leq
u(y_0)+3\Vinf{u}\left(\fr{1}{2^{m+1}}+1-\fr{1}{2^m}\right)
<u(y_0)+3\Vinf{u}.
\]

\item
If $y\not\in Y_k$ for all $k\in\N\cup\set{0}$,
then we have $u(y)=u(y_0)$ by \eqref{lem2.9.1},
and thus 
it follows from $\Vinf{v}=3\Vinf{u}$ that 
$u(y)+v(y)\leq u(y_0)+3\Vinf{u}$.
\end{itemize}
We derive from the above arguments that
$u(y)+v(y)\leq u(y_0)+3\Vinf{u}=u(y_0)+v(y_0)$,
which yields $u+v\in\PB{y_0}$.
\end{proof}

\section{Proof of main results}

We will keep 
 the notation 
in the last section.
Lemma~\ref{lem2.8} ensures that
there exists a well-defined map $\phi\colon Y\to X$
such that $S(\PB{y})=\PA{\phi(y)}$ for all $y\in Y$.
By the same manner, we can define a map
$\tau\colon X\to Y$ that satisfies 
$T(\PA{x})=\PB{\tau(x)}$ for all $x\in X$.

\begin{proof}[\textbf{Proof of Theorem~\ref{thm2}}]
Fixing arbitrary $f\in\pcx$ and $y\in Y$,
we shall prove that
\begin{equation}\label{lem3.1.1}
T(f)(y)=f(\phi(y))
\end{equation}
holds.
We first prove that $T(f)(y)\leq f(\phi(y))$.
There exists $v\in\PB{y}$ such that
$T(f)+v\in\PB{y}$ by Lemma~\ref{lem2.9}.
Then we obtain
$S(T(f)+v)\in\PA{\phi(y)}$,
since $S(\PB{y})=\PA{\phi(y)}$.
This implies that
\begin{align*}
S(T(f)+v)(\phi(y))
&=
\Vinf{S(T(f)+v)}
=\Vinf{T(f)+v}
=T(f)(y)+v(y),
\end{align*}
where we have used Lemma~\ref{lem2.3}.
Noting that $\TI=T^{-1}$ is an isometry,
we infer from the last equalities that
\begin{align*}
T(f)(y)+v(y)-f(\phi(y))
&=
\TI(T(f)+v)(\phi(y))-\TI(T(f))(\phi(y))\\
&\leq
\Vinf{\TI(T(f)+v)-\TI(T(f))}
=\Vinf{T(f)+v-T(f)}
=\Vinf{v}.
\end{align*}
We note that $v(y)=\Vinf{v}$,
since $v\in\PB{y}$.
The above inequalities prove that $T(f)(y)\leq f(\phi(y))$.

We may apply Lemma~\ref{lem2.9} to 
$f\in\pcx$ with $\phi(y)\in X$, 
and then there exists $g\in\PA{\phi(y)}$ such that
$f+g\in\PA{\phi(y)}$.
Then $f(\phi(y))+g(\phi(y))=\Vinf{f+g}$
and $T(f+g)\in T(\PA{\phi(y)})=T\left(\TI(\PB{y})\right)=\PB{y}$.
Lemma~\ref{lem2.3} assures that
\[
T(f+g)(y)=\Vinf{T(f+g)}=\Vinf{f+g}=f(\phi(y))+g(\phi(y)).
\]
We derive from the last equalities that
\begin{align*}
f(\phi(y))+g(\phi(y))-T(f)(y)
&=
T(f+g)(y)-T(f)(y)\\
&\leq
\Vinf{T(f+g)-T(f)}
=\Vinf{g},
\end{align*}
since $T$ is an isometry.
Keeping in mind that $g\in\PA{\phi(y)}$,
we get $g(\phi(y))=\Vinf{g}$,
and consequently, we obtain
$f(\phi(y))\leq T(f)(y)$ by the inequalities above.
We thus conclude that \eqref{lem3.1.1} holds.

We need to prove that
$\phi\colon Y\to X$ is a homeomorphism.
By applying the definitions of $\phi$ and $\tau$,
we obtain
$\TI(\PB{y})=\PA{\phi(y)}$
and $T(\PA{\phi(y)})=\PB{\tau(\phi(y))}$,
and consequently,
$\PB{y}=T(S(\PB{y}))=T(\PA{\phi(y)})=\PB{\tau(\phi(y))}$,
where we have used that $\TI=T^{-1}$.
Lemma~\ref{lem2.4}, applied to the last equalities,
shows that $y=\tau(\phi(y))$.
By a quite similar argument, we have
$x=\phi(\tau(x))$ for all $x\in X$.
These two equalities show that
$\phi$ and $\tau$ are both bijective maps
with $\phi^{-1}=\tau$.

We shall prove that $\phi\colon Y\to X$ is continuous.
Take an arbitrary open set $O$ in $X$.
Fixing $y_0\in\phi^{-1}(O)$ arbitrarily,
we can find $f_0\in\PA{\phi(y_0)}$
with $f_0(\phi(y_0))=2=\Vinf{f_0}$
and $f_0(X\setminus O)=\set{0}$.
Setting $U=\set{y'\in Y:T(f_0)(y')>1}$,
we see that $U$ is an open subset of $Y$.
We infer from \eqref{lem3.1.1} that
$T(f_0)(y_0)=f_0(\phi(y_0))=2$,
which implies that $y_0\in U$.
We obtain $f_0(\phi(y'))=T(f_0)(y')>1$
for each $y'\in U$
by \eqref{lem3.1.1}.
This shows that $\phi(U)\subset O$,
since $f_0(X\setminus O)=\set{0}$,
and consequently, $\phi^{-1}(O)$ is
an open set in $Y$.
We have proved that $\phi$ is continuous on $Y$.

We may apply the above arguments
to the pair of $(S,\tau)$, instead of  $(T,\phi)$.
Then we get that $S(u)(x)=u(\tau(x))$ holds
for all $u\in\pcy$ and $x\in X$, 
and thus, 
$\tau$ is continuous on $X$.
This proves that $\phi$ is a homeomorphism,
since $\tau=\phi^{-1}$.
\end{proof}

\begin{proof}[\textbf{Proof of Theorem~\ref{thm1}}]
Suppose that $T\colon\pcx\to\pcy$ is a surjective
phase-isometry.
Then $T$ is a surjective isometry
by Proposition~\ref{prop1.1}.
Applying Theorem~\ref{thm2} to $T$, 
there exists a homeomorphism $\phi\colon Y\to X$
such that $T(f)(y)=f(\phi(y))$ holds
for all $f\in\pcx$ and $y\in Y$.
\end{proof}

\begin{cor}\label{cor3.2}
If $T\colon\pcx\to\pcy$ is a surjective isometry,
then there exists a unique surjective  linear isometry 
$\TT\colon\cx\to\cy$ such that
$\TT=T$ on $\pcx$.
\end{cor}

\begin{proof}
There exists a homeomorphism $\phi\colon Y\to X$
such that $T(f)(y)=f(\phi(y))$ for all $f\in\pcx$
and $y\in Y$ by Theorem~\ref{thm2}.
We define $\TT\colon\cx\to\cy$ as
$\TT(h)=h\circ\phi$ for each $h\in\cx$.
It is routine to check that $\TT$ is a linear isometry with $\TT=T$ on $\pcx$.
We shall show that $\TT$ is surjective.
Taking an arbitrary $u\in\cy$,
we denote by $u^+$ and $u^-$ the positive
and negative parts of $u$, respectively.
Since $T$ is surjective, there exist
$f_1,f_2\in\pcx$ satisfying $T(f_1)=u^+$
and $T(f_2)=u^-$.
Noting that $f_1,f_2\in\pcx$ with $f_1-f_2\in\cx$,
we have
\[
\TT(f_1-f_2)=(f_1-f_2)\circ\phi
=f_1\circ\phi-f_2\circ\phi
=T(f_1)-T(f_2)
=u^+-u^-=u.
\]
We thus conclude that $\TT$ is surjective,
and hence $\TT$ is a surjective linear
isometry with $\TT=T$ on $\pcx$.

Suppose that $T'\colon\cx\to\cy$ is another surjective linear isometry that satisfies $T'=T$ on $\pcx$.
By the Banach--Stone theorem, there exist
a continuous function $\al\colon Y\to\set{-1, 1}$
and a homeomorphism $\varphi\colon Y\to X$
such that $T'(h)(y)=\al(y)h(\varphi(y))$ holds
for all $h\in\cx$ and $y\in Y$.
Having in mind that $T'(f)(y)=T(f)(y)=f(\phi(y))$ for  $f\in\pcx$ and $y\in Y$, 
we observe that $\al(y)=1$ and $\phi(y)=\varphi(y)$ for all $y\in Y$, and hence
$T'=\TT$ on $\cx$.
Thus, $\TT$ is the unique extension of $T$ to $\cx$.
\end{proof}

\begin{proof}[\textbf{Proof of Corollary~\ref{cor1}}]
This is a direct consequence of Corollary~\ref{cor3.2}
with Proposition~\ref{prop1.1}.
\end{proof}


\end{document}